\documentclass[a4paper,reqno,11pt,oneside]{amsart}

\usepackage{amsfonts,amssymb,latexsym,xspace,epsfig,graphics,color}
\usepackage{amsmath,enumerate,stmaryrd,xy}
\usepackage{bbm}

\numberwithin{figure}{section}
\usepackage{subfigure}
\usepackage{dsfont}
\usepackage{a4wide}
\usepackage{enumitem} 
\usepackage{cite}   

\makeatletter
\@namedef{subjclassname@2020}{
  \textup{2020} Mathematics Subject Classification}

\makeatother
\usepackage{tikz}
\usepackage{tkz-berge}
\usetikzlibrary{arrows,backgrounds,trees}
\usetikzlibrary{matrix,decorations.pathreplacing,positioning}
\usepackage{stmaryrd}
\newtheorem{theorem}{Theorem}[section]       
   
             \newtheorem{prop}[theorem]{Proposition}         \newtheorem{lemma}[theorem]{Lemma}
\newtheorem{corollary}[theorem]{Corollary}
\newtheorem{defn}[theorem]{Definition}
\newtheorem{remark}{Remark}[section]
\newtheorem{exa}{Example}[section]

\def\A{\mathcal{A}}
\def\T {\mathcal{T}}

\def\spann{\mathop {\rm span}\nolimits}
\def\rank {\mathop {\rm rank}\nolimits}

\def\Der {\mathop {\rm Der}\nolimits}
\def\dim {\mathop {\rm dim}\nolimits}
\def\charr {\mathop {\rm char}\nolimits}
\everymath{\displaystyle}

\bibdata{prob}
\bibstyle{alpha}

\title[]{Derivations of Evolution Algebras associated to graphs over a field of any characteristic}

\author{Tiago Reis and Paula Cadavid}

\address{Tiago Reis
\newline
 Universidade Federal do ABC, Avenida dos Estados, 5001-Bangu - Santo Andr\'e - SP, Brazil
 \newline
 Universidade Tecnológica Federal do Paraná, Av. Alberto Carazzai, 1640 - Cornélio Procópio - PR, Brazil \newline
e-mail: treis@utfpr.edu.br
}

\address{Paula Cadavid 
\newline
 Universidade Federal do ABC, Avenida dos Estados, 5001-Bangu - Santo Andr\'e - SP, Brazil
 \newline
e-mail: paula.cadavid@ufabc.edu.br, pacadavid@gmail.com}


\begin{document}
\maketitle
\begin{abstract} 
The space of derivations of finite dimensional evolution algebras associated to graphs over a field with characteristic zero has been completely characterized in the literature. In this work we generalize that characterization by describing the derivations of this class of  algebras for fields of any characteristic.
\end{abstract}

\section{Introduction}
The Evolution Algebras are a type of non-associative genetic  algebras which appeared some years ago as an algebraic way to mimic the self-reproduction of alleles in non-Mendelian genetics.  The  first reference of a Theory of Evolution Algebras is due to Tian and Vojtechovsky in \cite{tian1} who  state the first properties for these mathematical structures. Further on, Tian in \cite{tian} gave many connections of evolution algebras with other mathematical fields such as graph theory, group theory, Markov processes, dynamical systems and others.  We refer the reader to \cite{PMP, YPMP, PMP2, camacho/gomez/omirov/turdibaev/2013, COT, costa1, Elduque/Labra/2015, Elduque/Labra/2019,  farrukh} and references therein for an overview of recent results. An evolution algebra is defined as follows.
\begin{defn}\label{def:evolalg}
Let $\mathbb{K}$ be a field  and let $\A:=(\A ,\, \cdot\,)$ be a  $\mathbb{K}$-algebra. We say that $\A$ is an evolution algebra if it admits a finite basis $S:=\{e_1,\ldots, e_n\}$, such that
\begin{equation}\label{eq:ea}
\begin{array}{rcl}
e_i \cdot e_i & = & \displaystyle \sum_{k=1}^{n} w_{ik} e_k, \text{ for } i\in\{1,\ldots,n\},\\[.2cm]
e_i \cdot e_j & = & 0,\text{ for }i,j\in\{1,\ldots,n\}\text{ such that }i\neq j.
\end{array}
\end{equation} 
 \end{defn}
 
 \smallskip
 
 A basis $S$ of $\A$  satisfying \eqref{eq:ea} is called {\it natural}. The scalars $w_{ik} \in \mathbb{K}$, for $i,k\in \{1,\ldots,n\}$, are called the {\it structure constants} of $\mathcal{A}$ relative to $S$ and $M_{S}=(w_{ik})$ is called the {\it structure matrix} of $\A$ relative to $S$. In this work we shall describe the space of derivations of a particular class of evolution algebras.
 
\smallskip

\begin{defn}
Let $\A$ be a $\mathbb{K}$-(evolution) algebra. A derivation  $d$ of $\A$ is an element of $\mathcal{L}_{\mathbb{K}}(\A)$,
where $\mathcal{L}_{\mathbb{K}}(\A)$ is the space of $\mathbb{K}$ linear operators on $\A$, such that
$$d(u \cdot v)= d(u) \cdot  v +  u\cdot d(v)\,$$
for all $u,v \in \A$. The space of all derivations of the $\mathbb{K}$-(evolution) algebra $\A$ is denoted by $\Der_{\mathbb{K}}(\A)$.
\end{defn}
In \cite{tian} it is proved that if $\A$ is a  $\mathbb{K}$-algebra with structure matrix $M_S=(w_{ij})$, then $d \in \mathcal{L}_{\mathbb{K}}(\mathcal{A})$ such that 
\begin{equation}\label{eq:derexp}
d(e_i)=\sum_{k=1}^{n} d_{ik}e_k
\end{equation}
is a derivation of $\A$  if, and only if, it satisfies the following conditions:
\begin{alignat}{1}
w_{jk}d_{ij}+ w_{ik}d_{ji}=0,&\,\,\, \text{ for }i,j,k\in\{1,\ldots,n\}\text{ such that }i\neq j,\label{eq:der1}\\
\sum_{k=1} ^{n} w_{ik}d_{kj}=2w_{ij}d_{ii}, &\,\,\, \text { for }i,j\in\{1,\ldots,n\}.\label{eq:der2}
\end{alignat}
We observe that the previous statement is independent of $\charr(\mathbb{K}).$ However, if  $\charr(\mathbb{K})=2$, then clearly the equation (\ref{eq:der2}) becomes

\begin{alignat*}{1}
\sum_{k=1} ^{n} w_{ik}d_{kj}=0, &\,\,\, \text { for }i,j\in\{1,\ldots,n\}.\label{eq:der3}
\end{alignat*}

The  equations (\ref{eq:der1}) and  (\ref{eq:der2}) are the starting point for calculating the derivations of an evolution algebra $\A$. For some genetic algebras the derivations has  already been described in \cite{costa1,costa2,gonshor,gonzales,holgate,farrukh, peresi}.  In the case of an evolution algebra, a complete characterization of such space is still an open question. The references \cite{Alsarayreh/Qaralleh/Ahmad/2017, PMP2,YPMP, cardoso, COT,  Elduque/Labra/2019, farrukh, Qaralleh/Mukhamedov/2019} provide a list of some of the most recent works in this subject. In \cite{COT} the authors prove that the space of derivations of $n$-dimensional complex evolution algebras with non-singular matrices is zero, and they describe the space of derivations of evolution algebras with matrices of rank $n-1$. In \cite{Elduque/Labra/2019} the authors study the derivations of evolution algebras with non-singular matrices for the case of fields with any characteristic. Although the approaches considered in \cite{Alsarayreh/Qaralleh/Ahmad/2017, PMP3,YPMP,cardoso, Qaralleh/Mukhamedov/2019} are different, they provide a useful contribution to the field. While \cite{cardoso} gives a complete characterization for the space of derivation on two-dimensional evolution algebras, and \cite{Alsarayreh/Qaralleh/Ahmad/2017,farrukh} consider the case of three-dimensional solvable and finite-dimensional nilpotent evolution algebras, \cite{PMP3} do it for the case of evolution algebras associated to graphs over a field of zero characteristic. The approach stated in \cite{PMP3} rely only on the structural properties of the considered graph. In \cite{YPMP} the authors study the space o derivations of some non-degenerating irreducible finite-dimensional evolution algebras depending on the twin partition of an associated directed graph and, finally, \cite{Qaralleh/Mukhamedov/2019}, provides a description of the derivations of three dimensional evolution Volterra algebras.

In this work we contribute with this line of research by providing a complete characterization of the space of derivations of an evolution algebra associated to a graph over a field of any characteristic. Therefore, we generalize the results obtained by \cite{PMP3}. This paper is organized as follows. In Section 2, we review  definitions and notation of Graph Theory and Evolution Algebras and  we states preliminary results.  Section 3 is devoted to state our main results.

\section{Preliminary definitions and notation}

We start with some standard definitions and notation of Graph Theory. A finite graph $G$ with $n$ vertices is a pair $(V,E)$ where $V:=\{1,\ldots,n\}$ is the set of vertices and $E \subseteq \{(i,j)\in V\times V:i\leq j\}$ is the set of edges. If $(i,j)\in E$  or $(j,i)\in E$ we say that $i$ and $j$ are {\it neighbors}, and we denote the set of neighbors of a vertex $i$ by $\mathcal{N}(i)$. In general, given a subset $U\subset V$, we denote $\mathcal{N}(U):=\{j\in V: j\in \mathcal{N}(i) \text{ for some }i\in U\}$, and $U^c:=\{j\in V: j\notin U\}$. The {\it degree} of vertex $i$, denoted by $\deg(i)$, is the cardinality of the set $\mathcal{N}(i)$. The {\it adjacency matrix}  $A_G=(a_{ij})$ of $G$ is an $n\times n$ symmetric matrix such that $a_{ij}=1$ if $i\in \mathcal{N}(j)$, and $a_{ij}=0$ otherwise. Note that, for any $k \in V$, $\mathcal{N}(k):=\{\ell \in V: a_{k\ell}=1\}$. We say that vertices $i$ and $j$ of a graph $G$ are {\it twins} if they have exactly the same set of neighbors, i.e. $\mathcal{N}(i)=\mathcal{N}(j)$. We notice that by defining the relation $\sim_{t}$ on the set of vertices $V$ by $i\sim_{t} j$ whether $i$ and $j$ are twins, then $\sim_{t}$ is an equivalence relation. An equivalence class of the twin relation is referred to as a {\it twin class}. In other words, the twin class of a vertex $i$, denoted by $\T(i)$, is the set $\T(i)=\{j\in V:i \sim_{t} j\}$. The set of all twin classes of $G$ is referred to as the {\it twin partition} of $G$. A graph is {\it twin-free} if it has no twins. A  {\it path from $i$ to $j$} is a  finite sequence of vertices $i_0,i_1,i_2,\ldots,i_n$, such that $i_0=i$, $i_n=j$ and $i_{k+1}\in \mathcal{N}(i_{k})$ for all $k\in\{0,1,\ldots,n-1\}$. In this case we say that $n$ is the {\it length} of the path. If $i=j$ we say that the path is a {\it cycle}. If the length of a cycle is $1$ we say that the cycle is a {\it loop}.  The minimum length of the paths connecting the vertices $i$ and $j$ is called {\it distance} between them and we denote it by $\text{d}(i,j)$. All the graphs we consider are {\it connected}, i.e. for any $i,j\in V$ there exists a path from $i$ to $j$. We consider only graphs which are simple, i.e. without multiple edges or loops. 

\smallskip
The evolution algebra associated to a given graph $G$, and denoted by $\A(G)$, is defined by letting $w_{ij}= a_{ij}$ for any $i,j\in V$.

\smallskip
\begin{defn}\label{def:eagraph}
Let $G=(V,E)$ be a graph with adjacency matrix  $A_G=(a_{ij})$. The evolution algebra associated to $G$ is the algebra $\A(G)$ with natural basis $S=\{e_i: i\in V\}$, and relations

\[
 \begin{array}{ll}\displaystyle
e_i \cdot e_i = \sum_{k\in V} a_{ik} e_k, \text{  for  }i \in  V,\\[.5cm]
e_i \cdot e_j =0,\text{ if }i\neq j.
\end{array}
\] 
\end{defn}
\smallskip
\begin{exa}\label{ex:1} Let $G$ be the graph of Figure \ref{fig_ex:1}. The evolution algebra $\A(G)$ associated to $G$ has natural basis $\{e_{1},e_{2},e_{3},e_{4},e_{5},e_{6},e_{7}\}$ and relations given by
$$e^{2}_{1}=e^{2}_{2}= e_3,\,\,\,e^{2}_{5}=e^{2}_{6}= e^{2}_{7}=e_4, \,\,\, e^{2}_{3}=e_{1}+e_{2}+e_4, \,\,\, e^{2}_{4}=e_3+e_{5}+e_{6}+e_{7}\text{ and } e_{i}\cdot e_{j}=0, \text{ for } i\neq j.$$

\end{exa}

\begin{figure}[!h]
\centering
\begin{tikzpicture}[scale=1]
\draw (0,0) -- (-1,1);
\draw (0,0) -- (-1,-1);
\draw (0,0) -- (1,0) -- (2,1);
\draw (1,0) -- (2,0);
\draw (1,0) -- (2,-1);
\filldraw [black] (-1,1) circle (1pt);
\filldraw [black] (-1,-1) circle (1pt);
\filldraw [black] (1,0) circle (1pt);
\filldraw [black] (2,1) circle (1pt);
\filldraw [black] (2,0) circle (1pt);
\filldraw [black] (2,-1) circle (1pt);
\filldraw [black] (0,0) circle (1pt);

\node[left] at (-1,1) {$^{1}$};
\node[left] at (-1,-1) {$^{2}$};
\node[above] at (0,0) {$_{3}$};
\node[above] at (1,0) {$_{4}$};
\node[right] at (2,1) {$^{5}$};
\node[right] at (2,0) {$^{6}$};
\node[right] at (2,-1) {$^{7}$};
\end{tikzpicture}

\caption{Graph of Example \ref{ex:1}. }\label{fig_ex:1}
\end{figure}
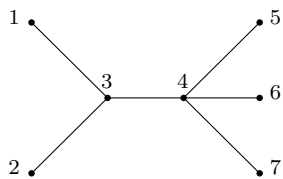
\smallskip
We refer the reader to \cite{PP,PMP,PMP2} for a review of recent results related to evolution algebras associated to graphs.
 
\bigskip
\subsection{Space of derivations of the evolution algebra of a graph}

Our aim is to characterize the space of derivations of an evolution algebra over fields of non-zero characteristic associated to a graph. In the rest of the paper we shall assume that $G=(V,E)$ is a finite graph with $n$ vertices and $d \in \Der_{\mathbb{K}}(\A(G))$ is definided as in \eqref{eq:derexp}. We will start with the following observation  related to  the simplest cases, which are when the graph has one or two vertices.

\smallskip
\begin{remark}\label{remark:1} 
Let $\mathbb{K}$ be a field. Let $G=(V,E)$ be a graph with adjacency matrix $A_G=(a_{ij})$. If $|V|=1$ as we consider only graphs without loops, we have $v\cdot u=0$ for all $u, v \in \A(G)$, and hence $\Der_{\mathbb{K}}(\A(G))=\mathcal{L}_{\mathbb{K}}(\A(G))$. If $|V|=2$, then since we are considering simple connected graphs, $G$ is necessarily the graph of Figure \ref{fig_remark:1}. Therefore, $a_{11}=a_{22}=0$ and $a_{12}=a_{21}=1$. If $d=(d_{ij})\in \Der_{\mathbb{K}}(\A(G))$,  taking $i=k=1$ and $j=2$ in \eqref{eq:der1} we have that $d_{12}=0$. Also, taking $i=k=2$ and $j=1$ we have $d_{21}=0$. Therefore $d$ is a diagonal matrix. Now, using \eqref{eq:der2} with $i\neq j$, we have 
\begin{equation}
    d_{11}=2d_{22} \text{ and } d_{22}=2d_{11}, \label{eq:V}
\end{equation}
which implies that $d_{11}=4d_{11}$ and $d_{22}=4d_{22}$. Hence, if $\charr (\mathbb{K}) \neq 3$, this implies that $d_{11}=d_{22}=0$ and therefore $ d=0$.  Otherwise, if  $\charr (\mathbb{K}) = 3$, from \eqref{eq:V} we conclude that  $d_{22}=2d_{11}$. Thus
\[\Der_{\mathbb{K}}(\A(G))= \left\{\left(\begin{matrix} \alpha & 0 \\ 0& 2\alpha\end{matrix}\right), \alpha \in \mathbb{K}
\right\}.\]

\end{remark}

\begin{figure}[!h]
\centering
\begin{tikzpicture}[scale=1]
\draw (0,0) -- (-1,0);
\filldraw [black] (0,0) circle (1pt);
\filldraw [black] (-1,0) circle (1pt);
\node[above] at (0,0) {$^{2}$};
\node[above] at (-1,0) {$^{1}$};
\end{tikzpicture}
\caption{Graph of Remark \ref{remark:1}.}\label{fig_remark:1}
\end{figure}
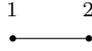
\smallskip

In what follows we will consider graphs with more than $2$ vertices, that is  $n\geq 3$. The next proposition provides
a look for the conditions \eqref{eq:der1} and  \eqref{eq:der2} in the context of evolution algebras associated to a graphs.

\smallskip
\begin{prop}\label{prop:conditions} Let $\mathbb{K}$ be a field with any characteristic and let $G=(V,E)$ be a  graph. Then $d\in \Der_{\mathbb{K}}\left(\A(G)\right)$ if, and only if, $d$ satisfies the following conditions:
\begin{enumerate}[label=(\roman*)]
\item \label{prop:conditions1} If $i,j\in V$, $i\neq j$, and $\mathcal{N}(i) \cap \mathcal{N}(j) \neq \emptyset,$ then $d_{ij}=-d_{ji}$.

\item \label{prop:conditions2} If $i,j\in V$, $i\neq j$, and $\mathcal{N}(i) \cap \mathcal{N}(j)^c \neq \emptyset,$ then $d_{ji}=d_{ij}=0$.
\item  \label{prop:conditions3} For any $i\in V$
\[
\sum_{k\in \mathcal{N}(i)} d_{kj}=\left\{
\begin{array}{cl}
0,&\text{ if }j\notin \mathcal{N}(i),\\[.2cm]
2d_{ii},&\text{ if }j\in \mathcal{N}(i).
\end{array}\right.
\]
\end{enumerate}
\end{prop}

\smallskip
\begin{proof}
If  $d\in \Der_{\mathbb{K}}\left(\A(G)\right)$ then the proof that $d$ satisfies conditions \ref{prop:conditions1} to \ref{prop:conditions3} is the same as the one  in \cite[Proposition 3.1]{PMP3}. Conversely, let $d \in \mathcal{L}_{\mathbb{K}}\left(\A(G)\right) $
satisfying the conditions \ref{prop:conditions1} to  \ref{prop:conditions3}. To prove that $d \in \Der_{\mathbb{K}}\left(\A(G)\right)$, we shall to check that $d$ verifies \eqref{eq:der1} and \eqref{eq:der2}. Let $i,j,k\in V$. To prove \eqref{eq:der1} we have to consider several cases: \newline
\smallskip
\noindent{\it Case 1.} If $k \in \mathcal{N}(i)^{c}\cap  \mathcal{N}(j)^{c}$, then  $a_{jk}=a_{ik}=0$. Therefore (\ref{eq:der1}) is true. \newline
\smallskip
\noindent{\it Case 2.} For $k \in \mathcal{N}(i)\cap \mathcal{N}(j)^{c}$ we  have that $a_{ik}=1$, $a_{jk}=0$ and $d_{ji}=0$, by  \ref{prop:conditions2}. Then (\ref{eq:der1}) it follows. For $k\in \mathcal{N}(i)^{c}\cap \mathcal{N}(j)$ the argument is analogous to that used in the previous situation.

\smallskip
\noindent{\it Case 3.} If $k \in \mathcal{N}(i)  \cap  \mathcal{N}(j)$, then $a_{jk}=a_{ik}=1$ and, by condition \ref{prop:conditions1},  $d_{ij}=-d_{ji}$. Thus we have $a_{jk}d_{ij}+ a_{ik}d_{ji}=d_{ij}-d_{ij}=0.$

\smallskip

To prove \eqref{eq:der2}, note that by \ref{prop:conditions3} we have that
\[ \begin{array}{rcl}
\sum_{k=1}^{n} a_{ik}d_{kj}
&= &\sum_{k \in \mathcal{N}(i)} d_{kj},\\
&=&\begin{cases}
0,&\text{ if }j\notin \mathcal{N}(i),\\
2d_{ii},&\text{ if }j\in \mathcal{N}(i), 
\end{cases} 
\\[.4cm]
&=&2a_{ij}d_{ii}.  \\
\end{array}
\]
for all $i,j,k\in V$.

\end{proof}

\smallskip
We point out that if $i,j \in V$  are such that  $i \not =j$ and $\mathcal{N}(i) \cap \mathcal{N}(j) =\emptyset $ then, as a consequence of \ref{prop:conditions2} in the above proposition ,  $d_{ij}=d_{ji}=0.$ On the other hand, if $\charr(\mathbb{K})=2$  then the condition \ref{prop:conditions3}  is written as
$$\sum_{k\in \mathcal{N}(i)} d_{kj}=0, \hspace{8pt} \text{ for all } i,j \in V.$$

The  above proposition extends the result in \cite[Proposition 3.1]{PMP3} where was proved that if the characteristic of the field is zero then the entries of a derivation satisfies the conditions
\ref{cor:conditions1}
to \ref{cor:conditions4}.
\smallskip

The statement of the following lemma, except for the characteristic of the field, is the same  as  \cite[Lemma 3.4]{PMP3}, which was  stated for fields with zero characteristic. But the same proof works for any characteristic. Then we will enunciate it without proof.

\smallskip

\begin{lemma}\label{lemma:algum}
Let $\mathbb{K}$ a field, let $G=(V,E)$ be graph,  and let $d\in \Der_{\mathbb{K}}\left(\A(G)\right)$. If $d_{ij}\neq 0$, for some $i,j\in V$ with $i\neq j$, then $i\sim_{t} j$.

\end{lemma}

An important result follows from this lemma:  if $i\not\sim_{t} j$ then $d_{ij}=0$. This shows that the matrix of a derivation is formed by blocks, according to the structure of the graph.

\begin{corollary}\label{cor:conditions}
Let  $\mathbb{K}$ be a field with $\charr(\mathbb{K})=p$, $p\not \in \{ 0,2\}$. Let  $G=(V,E)$ be a  graph and let $d\in \Der_{\mathbb{K}}\left(\A(G)\right)$. Then $d$ satisfies the following conditions:
\begin{enumerate}[label=(\roman*)]
\item If $i,\ell \in V$ and $i\sim_{t} \ell,$ then $d_{ii}=d_{\ell \ell}$.\label{cor:conditions1}
\item If $i\in V $ and $p \mid \deg(i)$, then
\begin{equation}
\sum_{k\in\mathcal{N}(i)}d_{kk} =0.
\end{equation} \label{cor:conditions2}
\item If $i\in V$ and $p \nmid \deg(i)$, then
\begin{equation}
2 \deg(i)d_{ii}= \sum_{k\in\mathcal{N}(i)}d_{kk}.
\end{equation} \label{cor:conditions3}
\item  Let $\ell \in V$. Then 
$\displaystyle \sum_{k \in \T(\ell)} d_{kj}=2d_{ii}$ for any $j \in \T(\ell)$ and $i\in \mathcal{N}(\T(\ell))$. \label{cor:conditions4}
\end{enumerate}
\end{corollary}

\begin{proof}
If $i\sim_{t} \ell$ and $j \in \mathcal{N}(i)=\mathcal{N}(\ell)$, then  Proposition \ref{prop:conditions}\ref{prop:conditions3} provides that
$$2d_{ii}=\sum_{k\in \mathcal{N}(i)}d_{kj}=\sum_{k\in \mathcal{N}(\ell)}d_{kj}=2d_{\ell \ell},$$
which proves (i). On the other hand, if $i,j\in V$ and $j\in\mathcal{N}(i)$, by Proposition \ref{prop:conditions}\ref{prop:conditions3}, we have that
\begin{equation*}\label{eq:cor1}
2d_{ii}=\sum_{k\in\mathcal{N}(i)}d_{kj}.  
\end{equation*}
Since the  equation above holds for any $j\in \mathcal{N}(i)$, then
$$
\sum_{j\in\mathcal{N}(i)}2d_{ii}=\sum_{j\in\mathcal{N}(i)}\sum_{k\in\mathcal{N}(i)} d_{kj},
$$
that we can rewrite as follows
\begin{equation}\label{eq:cor2}
2 \deg(i) d_{ii} = \sum_{k\in\mathcal{N}(i)} d_{kk} + \sum_{\substack{j,k \in \mathcal{N}(i)\\j\neq k}}d_{kj}.
\end{equation}
We note that, by Proposition \ref{prop:conditions}(i), it is true that $$\displaystyle  \sum_{\substack{j,k \in \mathcal{N}(i)\\j\neq k}}d_{kj}=0.$$

Now we shall analyze two situations. If  $p \mid \deg(i)$, then \eqref{eq:cor2}  becomes 

$$\displaystyle \sum_{k\in\mathcal{N}(i)} d_{kk} =0,$$

which proves (ii). Otherwise, if $p \nmid \deg(i)$, then by \eqref{eq:cor2} we have that $$ 2\deg(i) d_{ii} = \sum_{k\in\mathcal{N}(i)} d_{kk},$$ which proves (iii). Let $i,j\in V $ such that $j \in \T(\ell)$ and $i \in \mathcal{N}(\T(\ell))$. Then $\T(\ell) \subset \mathcal{N}(i)$ and 
  $$\displaystyle \sum_{k \in \T(\ell)} d_{kj}=\sum_{k \in  \mathcal{N}(i)} d_{kj} - \sum_{k \in \mathcal{N}(i) \backslash \T(\ell)} d_{kj}.$$
But, if $k \in \mathcal{N}(i) \backslash \T(\ell)$, then $k \not\sim_{t} j$ and, by the previous lemma, $d_{kj}=0$. Also, by the Proposition \ref{prop:conditions}\ref{prop:conditions3} we have 
$$\displaystyle \sum_{k \in \T(\ell)} d_{kj}=\sum_{k \in \mathcal{N}(i)} d_{kj}=2d_{ii},$$
which proves \ref{cor:conditions4}.

\end{proof}

\smallskip

\begin{corollary}\label{cor:conditionsp2}

Let $\mathbb{K}$ be a field with $\charr(\mathbb{K})=2$. Let $G=(V,E)$ be a graph and let $d\in \Der_{\mathbb{K}}\left(\A(G)\right)$. Then 
 $$\displaystyle \sum_{k \in \T(\ell)} d_{kj}=0, \textrm{ for any }  \ell \in V, j \in \T(\ell) \textrm{ and } i\in \mathcal{N}(\T(\ell)).$$ 
\end{corollary}
\smallskip

\begin{proof}
We can now proceed analogously to the proof of Corollary  \ref{cor:conditions} \ref{cor:conditions4}. Let $i,j\in V $
  such that $j \in \T(\ell)$ and $i \in \mathcal{N}(\T(\ell))$. Then $\T(\ell) \subset \mathcal{N}(i)$
 and therefore
$$\displaystyle 
\sum_{k \in \T(\ell)} d_{kj}=\sum_{k \in \mathcal{N}(i)} d_{kj}-\sum_{k \in \mathcal{N}(i) \backslash \T(\ell)} d_{kj} =  \sum_{k \in \mathcal{N}(i)} d_{kj}. $$
Now, using Proposition \ref{prop:conditions}\ref{prop:conditions3} we have
$$\displaystyle \sum_{k \in \T(\ell)} d_{kj}=\sum_{k \in \mathcal{N}(i)} d_{kj}=0.$$

\end{proof}

\begin{lemma}\label{prop:necessarylemma}Let  $\mathbb{K}$ be a field with $\charr(\mathbb{K})=p$, $p\not =2$. Let  $G=(V,E)$ be a graph  and let $d\in \Der_{\mathbb{K}}\left(\A(G)\right)$.
 If $d_{k\ell}\neq 0$ for some $k,\ell\in V$, with $k\neq \ell$, then $|\T(\ell) |\geq 3$.
\end{lemma}

\begin{proof}
Without loss of generality assume $d_{12}\neq 0$. By Lemma \ref{lemma:algum}, $d_{12}\neq 0$ implies that $1\sim _{t}2 $, that is $1,2 \in \T(1)$. Therefore $|\T(1)|\geq 2$. Assume on the contrary that $|\T(1)|=2$. Let $m \in \mathcal{N} (1)=\mathcal{N}(2)$. By Proposition \ref{prop:conditions}\ref{prop:conditions3} we have that
$$\displaystyle 2d_{mm}=\sum_{j \in \mathcal{N}(m)} d_{j1}=d_{11}+d_{21}+\sum_{\substack{j \in \mathcal{N}(m)\\ j \neq 1,2}} d_{j1}.$$
By hypothesis, if  $j \in \mathcal{N}(m)\setminus \{1,2 \}$ then $j\not\sim_{t} 1$. Therefore, by Lemma \ref{lemma:algum}, $d_{j1}=0$. Thus,
\begin{equation} \displaystyle
2d_{mm}=d_{11}+d_{21}+\sum_{\substack{j \in \mathcal{N}(m)\\ j \neq 1,2 }} d_{j1}= d_{11}+d_{21}. \label{d1j}
\end{equation}
Analogously,
\begin{equation}  \displaystyle
2d_{mm}=\sum_{j \in \mathcal{N}(m)} d_{j2}=d_{12}+d_{22}+\sum_{\substack{j \in \mathcal{N}(m)\\ j \neq 1,2}} d_{j2}=d_{12}+d_{22}.\label{d2j}
\end{equation}
Then, by (\ref{d1j}) and (\ref{d2j}),  $d_{11}+d_{12} = d_{21}+d_{22}.$
But, by Corollary \ref{cor:conditions} \ref{cor:conditions1},  $d_{11} = d_{22}$. Therefore, $d_{12}=d_{21}$, and this contradicts Proposition \ref{prop:conditions} \ref{prop:conditions1}. This completes the proof.

\end{proof}

\smallskip
We point out that the previous lemma generalizes \cite[Lemma 3.5]{PMP3}. Moreover, our proof is much simpler  than it.

\smallskip
 
The previous lemma is not true if $p = 2$ as we will see in the following example.

\begin{exa} \label{ex:DerF2} Let $\mathbb{F}_2$ be the  Galois Field with $2$ elements and let  $G$ be a graph of  Figure \ref{fig_ex:1}. Direct computations show that the operator $d$ given by 
$$d= \left( \begin{array}{ccccccc}
1&1&0&0&0&0&0 \\
1&1&0&0&0&0&0 \\
0&0&0&0&0&0&0 \\
0&0&0&0&0&0&0 \\
0&0&0&0&1&1&0 \\
0&0&0&0&1&1&0 \\
0&0&0&0&0&0&0 \\
\end{array} \right)$$ 
is an element of $\Der_{\mathbb{K}}\mathcal{A}(G)$. Note that $|\T(1)|=2$ and $d_{12}\neq 0$. Also, note that $6 \sim_{t} 7$ and $d_{66} \neq d_{77}$, which show that Corollary \ref{cor:conditions}\ref{cor:conditions1} does not apply to $\charr{\mathbb{K}}=2$. 
\end{exa}

\bigskip
\section{Main Results}
In this section, we present the  complete characterization of the space  of  derivations  of  finite dimensional evolution algebras associated to graphs over a field with any  characteristic. Since this characterization is valid to zero characteristic, it constitutes a generalization of \cite[Theorem 2.6]{PMP3}. Before enunciating our main result, let us make some observations and fix notation.

\smallskip

Let $G=(V,E)$ be a graph, let  $B=\{e_{1}', \dots, e_{n}'\}$ be a natural basis for $\A(G)$ and  let $\{c_1, \dots, c_m \}$ be a complete set of representatives of $\sim_{t}$. Let
$$t_0:=0 \,\,\,\,\,\text { and }\,\,\,\,\, t_i := \displaystyle \sum_{\ell=1}^i|\T(c_{\ell})|,\,\,\,\, \text{ for } i =1, \dots, m.$$
Let  $\sigma: V \to V $ be  a permutation of $V$ such that if  $i\in \T(c_k)$  and $ \sigma(i)=j$ then  $t_{k-1} < j \leq t_k.$

Clearly, although it is not unique, there is at least one $\sigma \in S_{n}$ with this property.
Let $$e_{\sigma (i)}:=e'_{i}, $$ 
Then, using $\sigma$  we can label the vertices of $G$ in such a way that $\{t_{i-1}+1,\dots , t_i \}= \T(c_i) $,  for $i\in \{1,\ldots, m \}$; that is, in such a way that the elements of every twin class are consecutive.
Equivalently, using $\sigma$ we can reorder  the elements of $B$ in such a way that 
\[B=\displaystyle\{\underbrace{e_1, \dots, e_{t_1}}_{|\T(c_1)|},\underbrace{e_{t_1+1}, \dots , e_{t_2}}_{|\T(c_2)|}, \dots,\underbrace{e_{t_{m-1}+1}, \dots, e_{t_m}}_{|\T(c_m)|} \}.\]
Therefore, without lost of generality, we will assume in the remainder of this section  that the vertices of $G$ are labeled in such way that the elements of a twin class are consecutive. 
\smallskip

Let $C=(c_{ij}) \in M_n(\mathbb{K})$. We  will denote by $\overline{C}=( {\overline c}_{ij})$ the matrix such that

\[{\overline c}_{ij}:= \left \{ \begin{array}{rl} 
c_{ij}, & \text{ if } i \not= j,\\
0, & \text{ if } i = j.\\
\end{array} \right. \]

The next theorem, which is the main result of this work and  gives a complete characterization for the derivations of an evolution algebra associated to a graph.

\smallskip

\begin{theorem}\label{teo:principal} Let $\mathbb{K}$ be a field such that $\charr(\mathbb{K})\not= 0$. Let $G=(V,E)$ be a graph  and  let $\{c_1, \dots, c_m \}$ be  a complete set of representatives of the twin classes. Then  $d \in \mathcal{L}_{\mathbb{K}}\left(\A(G)\right)$ is also in $\Der_{\mathbb{K}}\left(\A(G)\right)$ if, and only if, satisfies the following conditions:
\begin{enumerate}[label=(\roman*)]
\item  \label{teo:principal1} $d$ is a block matrix with the form
$$\displaystyle d=\left( \begin{array}{cccc}
     C_1&0&\dots&0   \\
     0&C_2& \dots &0 \\
     \vdots&\vdots &\ddots &\vdots \\
     0&0&\dots &C_m
\end{array} \right),$$
where $C_i \in M_{u_{i}}(\mathbb{K}),$ for  $i \in \{1,\dots, m \}$ and $u_{i}:=|\T(c_i)|$.

\smallskip

\item \label{teo:principal2} If $j \in \T(c_i)$, with $i \in \{1,\dots, m \},$ then $$\displaystyle \sum_{k\in \T(c_i)} d_{kj}=2d_{tt}, \text{ for all } t \in \mathcal{N}(\T(c_i)).$$ 

\item \label{teo:principal3}  $\overline{C_i}$, for $i \in \{1,\dots, m \}$,  is skew-symmetric. 
\end{enumerate}
\end{theorem}

\begin{proof} Let $t_0:=0$  and 
$$ t_i := \displaystyle \sum_{\ell=1}^i|\T(c_{\ell})|, \,\,\,\,\, i\in \{ 1, \ldots, m\}. $$ 
We can assume that the vertex of $G$ are labeled in such a way that $\{t_{j-1}+1,\dots , t_j \}= \T(c_j)$,
for $j\in \{ 1, \ldots, m\}.$ Let $d\in \Der_{\mathbb{K}}(\A(G))$.
If  $i,k \in V$  are such that $i \in \T(c_j)$ and $k \in \T(c_j)^{c}$, for some  $j \in \{1, \dots, m\}$ then $d_{ik}=d_{ki}=0,$ by Lemma \ref{lemma:algum}. Defining
$$\displaystyle C_i:=\left( \begin{array}{cccc}
     d_{t_{i-1}+1,t_{i-1}+1}&d_{t_{i-1}+1,t_{i-1}+2}&\dots&d_{t_{i-1}+1,t_{i}}   \\
     d_{t_{i-1}+2,t_{i-1}+1}&d_{t_{i-1}+2,t_{i-1}+2}& \dots &d_{t_{i-1}+2,t_i} \\
     \vdots&\vdots &\ddots &\vdots \\
     d_{t_i,t_{i-1}+1}&d_{t_i,t_{i-1}+2}&\dots &d_{t_{i},t_{i}}
\end{array} \right), $$
for $i\in \{1,\ldots, m\}$
 we have that $d$ has the form required in \ref{teo:principal1}.

The statement \ref{teo:principal2} follows from Corollary \ref{cor:conditions} \ref{cor:conditions4}, for $p \neq 2$, and from the  Corollary  \ref{cor:conditionsp2}, for  $p=2$. Finally, the  statement \ref{teo:principal3} follows from Proposition \ref{prop:conditions} \ref{prop:conditions1}.

Conversely, we will assume that $d$ satisfies the conditions (i) to (iii). To prove that  $d\in \Der_{\mathbb{K}}\left(\A(G)\right)$, we will see that $d$ verifies the conditions of Proposition \ref{prop:conditions}. First, we claim that if $d_{ij}\not= 0$ then  $i\sim_{t} j$. In fact, by the blocks structure of $d$, if $d_{ij} \neq 0$ then $d_{ij} $ is an entry of $C_{k}$,  for some $k \in \{1, \dots, m \}$. Consequently,
$$\displaystyle    t_{k-1} <i \leq t_k \text{ and } t_{k-1} <j \leq t_k,$$ and therefore $i,j \in \T(c_k)$. 
Let $i,j\in V$, with $i\neq j$. If  $\mathcal{N}(i) = \mathcal{N}(j)$ then $i \sim_{t} j$ and  by our assumption (iii), we have $d_{ij}=-d_{ji}$. Otherwise, if $\mathcal{N}(i) \neq \mathcal{N}(j)$ then $i \not\sim_{t} j$ and so $d_{ij}=d_{ji}=0$, which proves that $d$ satisfies conditions  \ref{prop:conditions1} and \ref{prop:conditions2} of  Proposition \ref{prop:conditions}. Finally, to prove 
the last condition, we will consider two cases. Let $i,j \in V$.

\smallskip
\noindent {\it Case 1.} If $j \not\in \mathcal{N}(i)$, then  $j \not\sim_{t} k $, for all $k \in \mathcal{N}(i)$ and therefore  $d_{kj}=0$. Thus 
$ \displaystyle \sum_{k\in \mathcal{N}(i)} d_{kj}= 0.$

\smallskip
\noindent{\it Case 2.} If $j \in \mathcal{N}(i)$, then  
$$ \displaystyle \sum_{k\in \mathcal{N}(i)} d_{kj}= \sum_{k\in \mathcal{N}(i)\backslash\T(c_s)} d_{kj}+ \sum_{k\in \T(c_s)} d_{kj}= \sum_{k\in \T(c_s)} d_{kj}.$$
where $s \in \{1, \dots,m \}$ is such that  $j\in \T(c_s)$. Our assumption \ref{teo:principal2} implies that 
  $$\displaystyle \sum_{k\in \T(c_s)} d_{kj}=2d_{tt},\text{ for all } t\in \mathcal{N}(\T(c_s)).$$
Choosing $t= i$ we have the desired equality.

\end{proof}

\begin{remark}
We emphasize that the claim of the previous theorem is also true if the characteristic of the field is zero. In fact, if $d \in \Der_{\mathbb{K}}(\A(G))$ and   $\charr(\mathbb{K})=0$, then we have by \cite[Theorem 2.6]{PMP3} that $d$ satisfies the conditions \ref{teo:principal1} to \ref{teo:principal3}.  On the other hand, the arguments used to prove the other implication, are independent of the characteristic of $\mathbb{K}$.

\end{remark}

\begin{corollary}\label{corol:p>3} 

Let $\mathbb{K}$ be a field such that $\charr(\mathbb{K})= p$, with $p\not \in \{0,3\}$ and  let $G=(V,E)$ be a graph. If
$d\in \Der\left(\A(G)\right)$ is such that  $d_{ij}=d_{ji}=0$ for any $i,j\in V$, $i\neq j$, then $d_{ii}=0$ for any $i\in V$. 
\end{corollary}

\begin{proof} Let $i,j \in V$ such that  $j\in \mathcal{N}(i)$. First we assume that $p=2$.  By Proposition \ref{prop:conditions}\ref{prop:conditions3} we have
$$0=\sum_{k \in \mathcal{N}(j)} d_{ki}=\sum_{\substack{k \in \mathcal{N}(j)\\ k \neq i}} d_{ki}+d_{ii}=d_{ii}.$$

If $p>3$, by Theorem \ref{teo:principal}\ref{teo:principal2}, we have that 
 $$2d_{ii}=\displaystyle \sum_{k \in \T (i)}d_{kj}=d_{jj}\,\,\,\,\, \text{ and }\,\,\, \,\, 2d_{jj}=\displaystyle \sum_{k \in \T (j)}d_{ki}=d_{ii}.$$ 
Then $4d_{ii}=d_{ii}$ which implies that $d_{ii}=0$.

\end{proof}

\smallskip

Note that this result give us that there is not a diagonal derivation if $p\neq 0,3$. In \cite{PMP3} it has been proved that it is also valid for $p=0$, then the only possible exception is for $p=3$.

\begin{exa} Let  $\mathbb{K}$  be a field with  $\charr(\mathbb{K})=3$ and  let $K_{23}$ be the complete bipartite graph (see Figure \ref{fig2}). If $d\in \Der(\A(K_{23}))$, then by Theorem \ref{teo:principal}\ref{teo:principal1} and \ref{teo:principal3}  we have that
$$ d=\left(\begin{array}{ccccc}
    d_{11}&d_{12}&d_{13}&0&0\\
    -d_{12}&d_{22}&d_{23}&0&0\\
    -d_{13}&-d_{23}&d_{33}&0&0 \\
    0&0&0&d_{44}&d_{45}\\
    0&0&0&-d_{45}&d_{55}
\end{array} \right). $$
On the other hand, by Theorem \ref{teo:principal}\ref{teo:principal2}, we have 
\begin{eqnarray*}
d_{11}-d_{12}-d_{13}&=&2d_{44}=2d_{55} ;\\
d_{12}+d_{22}-d_{23}&=&2d_{44}=2d_{55} ;\\
d_{13}+d_{23}+d_{33}&=&2d_{44}=2d_{55}; \\
d_{44}-d_{45}&=&2d_{11}=2d_{22}=2d_{33} ;\\
d_{45}+d_{55}&=&2d_{11}=2d_{22}=2d_{33}.
\end{eqnarray*}
Defining $\alpha:= d_{44}$ and $\beta :=d_{12}$  we  can rewrite $d$ as
$$d=\left(\begin{array}{ccccc}
    2\alpha &\beta&-\beta&0&0\\
    -\beta&2\alpha&\beta&0&0\\
    \beta&-\beta&2\alpha&0&0 \\
    0&0&0&\alpha&0\\
    0&0&0&0&\alpha
\end{array}{} \right). $$
If $\alpha =1$ and $\beta = 0$, then $d$ is a block matrix, which implies that Lemma \ref{corol:p>3} indeed doesn't work for $p=3$.
\end{exa}
\smallskip
\begin{figure}[!h]\label{fig2}
    \centering
\begin{tikzpicture}[scale=1]
\draw (-1,-1) -- (-2,1);
\draw (-1,-1) -- (0,1);
\draw (-1,-1) -- (2,1);
\draw (1,-1) -- (-2,1);
\draw (1,-1) -- (0,1);
\draw (1,-1) -- (2,1);

\filldraw [black] (0,1) circle (1pt);
\filldraw [black] (-1,-1) circle (1pt);
\filldraw [black] (1,-1) circle (1pt);
\filldraw [black] (2,1) circle (1pt);
\filldraw [black] (-2,1) circle (1pt);
\node[above] at (-2,1) {$1$};
\node[above] at (0,1) {$2$};
\node[above] at (2,1) {$3$};
\node[below] at (-1,-1) {$4$};
\node[below] at (1,-1) {$5$};
\end{tikzpicture}
\caption{Complete bipartite graph $K_{23}.$}
\end{figure}
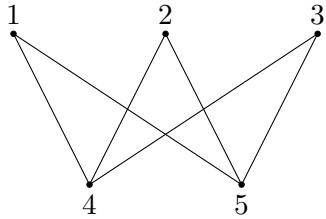

\smallskip

\begin{theorem}

 Let $\mathbb{K}$ be a field. Let $G=(V,E)$ be a  graph such that $\A(G)^2=\A(G)$. Then $ \Der(\A(G))$  satisfies the following properties:
\begin{enumerate}[label=(\roman*)]
\item \label{prop:elduque1} If $\charr (\mathbb{K})\neq 3,$ then $\Der_{\mathbb{K}}(\A(G))=0$.
\item \label{prop:elduque2} If $\charr (\mathbb{K})= 3,$ then $\dim_{\mathbb{K}}\Der_{\mathbb{K}}(\A(G)) \leq 1$ and the equality holds if, and only if, $G$ has no cycles of odd length. In this case $\Der_{\mathbb{K}}(\A(G))=\spann_{\mathbb{K}}\{ f \}$, where $f$ is the diagonal linear map $f=(f_{ij})$ defined by
\begin{equation}\label{fdim1}
f_{ii}:=\left\{
\begin{array}{cl}
1,&\text{ if } i=1,\\[.2cm]
2,&\text{ if } \text{d}(i,1)  \text{ is odd,}  \\[.2cm]
1,&\text{ if } \text{d}(i,1) \text{ is even.}  \\[.2cm]
\end{array}\right.  
\end{equation}

\end{enumerate}
\end{theorem}
\begin{proof}  We point out that the  assumption $\A(G)^2=\A(G)$ is equivalent  to say  that the structure matrix $A_G=(a_{ij})$ of $\A(G)$ is non-singular. Note that $i \not \sim_{t}  j$, for all $i,j \in V$. In other case, if $i \sim_{t} j $ for some $i, j \in V$ with $i\neq j$, then the rows $i$ and $j$ of $A_G$ are equal, thus $\rank(A_G)\neq n$. 

Let $d=(d_{ij}) \in \Der_{\mathbb{K}}(A(G))$. If $\charr (\mathbb{K})=0$, then $d=0$ by \cite[Theorem 2.3]{PMP3} (or by  \cite[Theorem 2.1]{COT}). If $\charr (\mathbb{K}) \not\in \{0, 3 \}$, by Lemma \ref{lemma:algum}, we have $d_{ij}=0$ for all $i,j \in V$  with $i\neq j$. Furthermore, by Corollary \ref{corol:p>3}, $d=0$, which proves \ref{prop:elduque1}.

Since $G$ is connected, for $i,k \in V$,
 there exist a path $k=j_{0},j_{1},j_{2}, \dots, j_{t}=i$ connecting $k$ to $i$. Then, by Theorem \ref{teo:principal} \ref{teo:principal2}, we have that
\begin{equation}\label{eq:caminho}
d_{kk}=2d_{j_1j_1}=d_{j_2j_2}=2d_{j_3j_3}=\dots = \left\{
\begin{array}{cl}
2d_{ii},&\text{ if } t \text{ is odd},\\[.2cm]
d_{ii},&\text{ if } t \text{ is even}.
\end{array}\right.
\end{equation}
Thus if $d_{jj}=0$ for some, $j \in V$, then $d=0$.

Suppose that $ \Der_{\mathbb{K}}(\A(G)) \neq 0$. Let  $d,d'\in\Der_{\mathbb{K}}(\A(G))$ such that both are not null  and   $d\neq d'$. Let $i\in V$. Then since $d_{ii}\not =0$ and $d'_{ii}\not =0$  there is $\lambda \in \mathbb{K}$, $\lambda \not =0$, such that $d'_{ii}=\lambda d_{ii}$. On the other hand,  for $k\in V$ there is a  path $k=j_{0},j_{1},j_{2}, \dots, j_{t}=i$, and  then by (\ref{eq:caminho}), we have 
$$d'_{kk} = \left\{
\begin{array}{cl}
2d'_{ii},&\text{ if } t \text{ is odd},\\[.2cm]
d'_{ii},&\text{ if } t \text{ is even}.
\end{array}\right.
=\left\{
\begin{array}{cl}
2\lambda d_{ii},&\text{ if } t \text{ is odd},\\[.2cm]
\lambda d_{ii},&\text{ if } t \text{ is even}.
\end{array}\right.= \lambda d_{kk}.
$$
Therefore $d'=\lambda d$, that is $\dim_{\mathbb{K}}\Der(\A(G)) = 1$.

Suppose that $G$ has a cycle of odd length $k=j_{0},j_{1},j_{2}, \dots , j_{t}=k$. Then, by (\ref{eq:caminho}), $d_{kk}=2d_{kk}$ and therefore $d=0$. Conversely, if $G$ has no cycles of odd length, consider the diagonal linear map $f=(f_{ij})$  defined in (\ref{fdim1}).

Let $j, \ell, t \in V$ be such that $j\in \T(\ell)$ and $t \in \mathcal{N}(\T(\ell))$. If $f_{jj}=2$, then $\text{d}(j,1)$ is odd, and hence  $\text{d}(t,1)$ is even and thus $f_{t t }=1$, otherwise we would have a cycle of odd length in $1$, which contraries the hypothesis. Thus
$$\displaystyle\sum_{k \in \T(\ell)}f_{kj}= f_{jj}=2f_{t t},$$
which proves that $f$ satisfies Theorem \ref{teo:principal}\ref{teo:principal2}.  If $f_{jj}=2$, an analogous argument provides the same result. Is not difficult to see that $f$  satisfies conditions \ref{teo:principal1} and \ref{teo:principal3} of Theorem \ref{teo:principal} so $f\in \Der_{\mathbb{K}}(\A(G))$ and therefore $\dim_{\mathbb{K}}\Der_{\mathbb{K}}(\A(G))= 1$.
\end{proof}

In \cite[Theorem 4.1]{Elduque/Labra/2019}  the authors describes the dimension of the  space of derivations of an evolution algebra  with non singular structure matrix using an oriented graph associated to a algebra and the balance of that graph. The previous theorem  gives a more specific description of this space for the case of
an evolution algebra associated to a graph.

\begin{corollary}
Let $\mathbb{K}$ be a field such that $\charr(\mathbb{K})= p$ and  let $G=(V,E)$ be a graph. Then 
$\Der_{\mathbb{K}}\left(\A(G)\right)=0$ for the following cases.
 \begin{enumerate}[label=(\roman*)]
\item \label{corol:twinfree1} If $p \not\in \{2,3\}$ and $|\T(i)| \leq 2$ for all $i \in V$.
\item \label{corol:twinfree2} If $p = 2$ and $G$  is twin-free.
\end{enumerate}
 \end{corollary}
 \begin{proof}
If $p = 0$, the result follows from \cite[Thereom 2.3]{PMP3}. If $p \not\in \{0,2,3\} $ and $d\in \Der_{\mathbb{K}}\left(\A(G)\right)$ then by Lemma \ref{prop:necessarylemma}, $d_{ij}=0 $ for all $i,j\in V$ with $i\neq j$, which implies that $d=0$ by Corollary \ref{corol:p>3}. To prove \ref{corol:twinfree2}, note that if $d\in \Der_{\mathbb{K}}\left(\A(G)\right)$, by Theorem \ref{teo:principal} \ref{teo:principal1} we have $d_{ij}=0 $ for all $i,j\in V$ with $i\neq j$.
Therefore, by Corollary \ref{corol:p>3}, we have $d=0$.

 \end{proof}

\smallskip
\section*{Acknowledgements} The first author thanks the UTFPR for all support provided and the use of the CCCT-CP computer facilities.

\end{document}